\documentclass[a4paper,11pt]{article}
\usepackage{amsthm}
\usepackage[tbtags]{amsmath}

\usepackage{amsfonts}
\usepackage{amssymb}
\usepackage{verbatim}

\def\vf{\vspace {4 mm}}
\def\vt{\vspace{ 10 mm}}

\def\E{\mathbb{E}}

\def\eps{\epsilon}

\def\1{\mathbf{1}}

\def\tce{t_c + \eps}
\def\tce2{t_c + \frac{\eps}{2}}
\def\ER{Erd\H{o}s-R\'{e}nyi }
\def\erdos{Erd\H{o}s }
\def\renyi{R\'{e}nyi }
\def\BF{Bohman-Frieze }

\def\var{\text{var}}

\newtheorem{thm}{Theorem}
\newtheorem{lem}{Lemma}

\newtheorem{prop}{Proposition}
\newtheorem{conj}{Conjecture}

\title{The Bohman-Frieze Process Near Criticality}
\author{Mihyun Kang\footnote{Technische Universit\"at Graz. kang@math.tugraz.at. Supported in part by the Heisenberg programme of the Deutsche Forschungsgemeinschaft (KA 2748/2-1).}, Will Perkins\footnote{School of Mathematics, Georgia Institute of Technology. perkins@math.gatech.edu. Supported in part by NSF grant OISE-0730136 and an NSF Postdoctoral Fellowship.}, and Joel Spencer\footnote{Courant Institute of Mathematical Sciences, New York University. spencer@courant.nyu.edu.}}

\begin{document}

\maketitle

\begin{abstract}
The \ER process begins with an empty graph on n vertices, with edges added randomly one at a time to the graph. A classical result of \erdos and \renyi states that the \ER process undergoes a phase transition, which takes place when the number of edges reaches n/2 (we say at time 1) and a giant component emerges. Since this seminal work of \erdos and R\'{e}nyi, various random graph models have been introduced and studied. In this paper we study the \BF process, a simple modification of the \ER process.

The \BF process also begins with an empty graph on $n$ vertices. At each step two random edges are presented, and if the first edge would join two isolated vertices, it is added to a graph; otherwise the second edge is added. We present several new results on the phase transition of the \BF process.  We show that it has a qualitatively similar phase transition to the \ER process in terms of the size and structure of the components near the critical point.  We prove that all components at time $t_c-\eps$ (that is, when the number of edges are $(t_c-\eps)  n/2$) are trees or unicyclic components and that the largest component is of size $\Omega(\eps^{-2} \log n)$.  Further, at $t_c + \eps$, all components apart from the giant component are trees or unicyclic and the size of the second-largest component is $\Theta(\eps^{-2} \log n)$.  Each of these results corresponds to an analogous well-known result for the \ER process.   Our proof techniques include combinatorial arguments, the differential equation method for random processes, and the singularity analysis of the moment generating function for the susceptibility, which satisfies a quasi-linear partial differential equation.
\end{abstract}

\section{Introduction}

\erdos and \renyi began the study of random graphs in 1959 with their landmark paper \lq On the Evolution of Random Graphs\rq  \cite{erdős1960evolution}. They described the changes in the structure of a graph as it evolves from empty (no edges) to full (all edges present), with a uniformly chosen random edge added at each step.  There are three closely related models of the \ER random graph:
(i) the binomial random graph $G(n,p)$, the graph on $n$ vertices in which each edge is present independently with probability $p$;
(ii) the uniform random graph $G(n,m)$, which is a graph chosen uniformly from all graphs of $n$ vertices and $m$ edges;
(iii) the \ER random graph process, which starts with an empty graph on $n$ vertices and at each step a random edge selected uniformly and independently from the set of potential edges is added to the graph.
The models are related as follows: $G(n,p)$ and $G(n,m)$ are essential equivalent for the correct choice of $m$ and $p$, and the $m$th step of the random graph process has the distribution of $G(n,m)$.

The most striking result from \cite{erdős1960evolution} was the discovery of a \lq double jump\rq\ in the largest component size at $\frac{n}{2}$ edges.  With $\frac{cn}{2}$ edges, $c<1$ a constant, all connected components of the random graph are of size $O(\log n)$ with probability $1- o(1)$.  With $\frac{n}{2}$ edges, the largest component is of size $\Theta(n^{2/3})$.  And with  $\frac{cn}{2}$ edges, $c>1$, there is a \lq giant component\rq\ of size $\Theta(n)$ and all other components are of size $O(\log n)$.  This rapid transformation of the graph is called the \lq phase transition\rq\ of the random graph and is related to similar phase transitions in percolation models in mathematical physics.

The phase transition of the \ER random graph is now understood in great detail.  For a comprehensive account of the results see \cite{alon2000probabilistic}, \cite{bollobas2001random}, \cite{janson2000random}.  In particular, the correct scaling for the critical window is known, as well as the distribution of the size and structure of components at each period of the phase transition.

Motivated by the successes of studying the \ER random graph, researchers have studied many different modifications of the random graph process. These modifications range from the study of random graphs with a given degree sequence \cite{kang2008critical}, \cite{molloy1995critical}, \cite{molloy1998size}, to random planar graphs \cite{kang2010two}, \cite{mcdiarmid2005random}, to random simplicial complices \cite{linial2006homological}.

One particular modification proposed by Achlioptas was to use the \lq power of two choices\rq\ to create processes with different behavior.  In an Achlioptas process, two randomly sampled potential edges are presented at each step, then one of them is chosen according to a given rule and added to the graph.  The choice of edge can depend on which edges were added to the graph in previous steps.  Each different rule corresponds to a different random graph process.  The \ER process itself comes from the rule \lq always add the first edge\rq\ (or the rule \lq always add the second edge\rq).  Achlioptas asked whether there was a rule that pushed back the emergence of a giant component.  In \cite{bohman2001avoiding}, Bohman and Frieze answered his question in the affirmative by analyzing the following rule: add the first edge if it would join two vertices that did not intersect any of the previously presented edges; otherwise add the second edge.  Their work showed that the choice of two edges at each step of the process could indeed cause different behavior, and initiated two major threads of research into Achlioptas processes.

The first is concerned with testing the power and limits of Achlioptas processes. How much can we accelerate or delay the phase transition?  How long can we delay the formation of a Hamiltonian cycle?  These questions can be asked in the original context of an Achlioptas process or in the off-line case in which all pairs of edges are given first, then the choices are made.  The results often generalize to the case in which $k$ edges are presented at each step instead of 2. Recent results concern (i) the acceleration or delaying of the phase transition \cite{beveridge2007product},  \cite{bohman2001avoiding},  \cite{bohman2004avoidance},\cite{bohman2006phase},  \cite{flaxman2005embracing}; (ii) the avoidance of small subgraphs \cite{krivelevich2009avoiding}, \cite{mutze2011small}, \cite{prakash2009balanced}; (iii) the acceleration of the appearance of Hamiltonian cycles \cite{krivelevich2010hamiltonicity}.

The second thread of research involves choosing one fixed Achlioptas rule and exploring the fine details of its evolution.  We will focus on this second thread of research in this paper, and in particular concentrate on the \BF process, a process very similar to the process in \cite{bohman2001avoiding}.  The rule for what is now known as the \BF process is as follows: if the first edge would join two isolated vertices, add it; otherwise add the second edge.   The \BF process is a simple modification of the \ER process with dependence between the edges, and different behavior: the emergence of the giant component is in fact delayed.  The motivation for studying its finer behavior is to develop methods for analyzing random graph processes with dependence between the edges and to understand how universal the \ER phase transition is.  The \BF process is also shorthand for a much wider class of Achlioptas rules, so-called \lq bounded-size rules\rq\ introduced in \cite{spencer2007birth}.  In a bounded-size rule, the choice between the two edges can only depend on the sizes of the four components that may be connected, and all components of size $> K$ must be treated the same, for some fixed constant $K$.  Most of the results on the phase transition of the \BF process, in \cite{1005.4494},\cite{spencer2007birth},  and in this work, can be transferred in a straightforward way to all bounded-size rules.

While the phase transition of the \BF process is delayed, it shares many qualitative similarities to that of the \ER process.  In this paper we strengthen these qualitative similarities, presenting results on the size and structure of the largest and second-largest components near the critical point. We also prove an asymptotic formula for the fraction of vertices in components of a given size.  As in previous analysis of the \BF process, we use the differential equation method for random processes. But here, using the ordinary generating functions arising from the \BF process, we derive a quasi-linear partial differential equation that tracks key statistics of the process. Then we apply the singularity implicit functions theorem and the transfer theorem from the work of Flajolet and Odlyzko~\cite{flajolet1990singularity} and Flajolet and Sedgewick \cite{flajolet2009analytic}.

Before we present our main results, we introduce notation and concepts that will be used throughout the paper.

\section{Preliminaries}

Throughout the paper we consider the following version of the \BF process: at each step, we select two edges uniformly at random (with replacement) from the set of edges not yet present in the graph.  If the first edge would join two isolated vertices, we add it.  Otherwise, we add the second edge.

\paragraph{Notation.}

We write \lq with high probability\rq\ or \lq whp\rq\ if an event holds with probability  $\to 1$ as $n \to \infty$ (other authors sometime use `aas' or `asymptotically almost surely'). We use the standard asymptotic notation $O(\cdot), o( \cdot), \Theta(\cdot), \Omega(\cdot),$ and $\omega(\cdot)$.  We often combine the asymptotic notation with a statement about probability: whp $g_n = O(f(n))$  means that there exists a constant $K$ so that $g_n \le K f(n)$ whp; and  whp $g_n = o(f(n))$ means that for every constant $\delta>0$, $g_n < \delta f(n)$ whp.
We sometimes use asymptotic notation as $n \to \infty$ and $\eps \to 0$, where $\eps$ is independent of $n$: whp  $g(n,\eps)=O(f(n,\eps))$ means that there exist constants $K$ and $\eps_0$, such that for all $0<\eps<\eps_0$, $g(n,\eps) \le K f(n,\eps)$ with probability $\to 1$ as $n \to \infty$, where the convergence need not be uniform in $\eps$.

For $i,j\in \mathbb N$, $C_i$ will refer to the $i$th largest connected component of a graph.  $X_1(j)$ will denote the number of isolated vertices at step $j$ of the process, and similarly $X_i(j)$ will be the number of vertices in components of size $i$ at step $j$ of the process.  Our main results concern the size of structure of connected components in the random graphs. In terms of structure, a \lq simple\rq\ component is either a tree or unicyclic (if a component is simple and has $k$ vertices, then it has $k-1$ or $k$ edges and is a tree or unicyclic component respectively).  A \lq complex\rq\ component is a connected component with two or more cycles, or $\ge k+1$ edges if it has $k$ vertices.

The \ER random graph undergoes a phase transition around step $\frac{n}{2}$. So, we will parameterize time in a standard way
\begin{equation}
  t= \frac{\# \text{ of edges}}{ n/2}
\end{equation}
so $t$ is the average degree in the graph. We write $G(t)$ for a graph with $t \frac{n}{2}$ edges in a random graph process.  We denote the size of a component $C$ by $|C|$, so $|C_1(t)|$ is the size of the largest component of the random graph process at time $t$.  The critical point $t_c$ for the phase transition is defined so that for all fixed $t<t_c$, $|C_1(t)| = O(\log n)$ and for all fixed $t> t_c$, $|C_1(t)| = \Theta(n)$.  For the \ER process, $t_c=1$, while for the \BF process $t_c$ is approximately $1.176$ \cite{spencer2007birth}.

\paragraph{The Differential Equation Method.}

A key method used in studying random graph processes is the \lq Differential Equation Method\rq.  Wormald \cite{wormald1995differential}, \cite{wormald1999differential} has proved several general theorems about the method. The method was used by Bohman \cite{bohman2009triangle} and  by Bohman and Keevash  \cite{bohman2010early} to analyze the triangle-free process and the H-free process respectively. It has previously been used by Spencer and Wormald \cite{spencer2007birth}, Bohman and Kravitz \cite{bohman2006creating}, and Janson and Spencer \cite{1005.4494} in analyzing the phase transition of Achlioptas processes.

The method is used to show that a random quantity, $X_i(j)$ for example, is tightly concentrated under the right scaling around a deterministic function throughout the course of a random graph process (using sub- and supermartingales).  The deterministic function is the solution to a system of ordinary differential equations. Applying Wormald's theorem (e.g. Theorem 5.1 in  \cite{wormald1995differential}) on the differential equation method, Spencer and Wormald \cite{spencer2007birth}  show that in a bounded-size Achlioptas process, the proportion of vertices in components of size $i$, for fixed $i$, is concentrated around a deterministic function $x_i(t)$.
In Section \ref{ODEanalysis} we study this system of ordinary differential equations for the specific case of the \BF process.

\paragraph{Susceptibility.}

The susceptibility of a graph $G$ is defined as the expected component size of a randomly chosen vertex:

\begin{equation}
\label{suscepdefeq}
S_1(G) = \frac{1}{n} \sum_v |K(v)| = \frac{1}{n}\sum_{C_i} |C_i|^2,
\end{equation}
where the second sum is over all connected components and $K(v)$ denotes the component containing the vertex $v$.  The name \lq susceptibility\rq\ comes from the study of percolation on infinite lattices in mathematical physics (see, for example, \cite{grimmett1999percolation}).  There, the susceptibility is defined to be the expected size of the connected cluster containing the origin.

We can also define higher moments of the susceptibility:
\begin{equation}
S_k(G) = \frac{1}{n} \sum_v |K(v)|^k = \frac{1}{n}\sum_{C_i} |C_i|^{k+1}.
\end{equation}
Using the differential equation method, Spencer and Wormald \cite{spencer2007birth} showed that for $t < t_c$, whp $S_1(G_j) = s_1(\frac{2j}{n}) + o(1)$ for some deterministic function $s_1(t)$ which blows-up at $t= t_c$.  Janson and Spencer \cite{1005.4494} extend this proof to $S_k$ for any fixed $k$. The deterministic functions $s_k(t)$ are given as the solution to a system of ODE's.  We write these ODE's for the \BF process and analyze them in Section \ref{ODEanalysis}.  The blow-up point of each of $s_k$'s is $t= t_c$ \cite{spencer2007birth}.   We will write $S(j)$ for $S_1(G_j)$ and $s(t)$ for $s_1(t)$ if we are considering the susceptibility apart from the higher moments.

\paragraph{Singularity Analysis.}

Given an ordinary generating function $f(z):=\sum_{i\ge 0} f_i\, z^i$,  we sometimes use the notation $[z^i]\, f(z)$ to mean $f_i$. When $f_i\ge 0$ for all $i\ge 0$, by Pringsheim's theorem~\cite{flajolet2009analytic},  among the singularities of  $f(z)$ which are closest to the origin, there is one with a positive real value, which is called the dominant singularity $\rho$ of $f(z)$. In this paper we will deal with  generating functions $f(z)$ which have a singular expansion at $\rho$ of the form
$$f(z) = g(z)-h(z) (1-z/\rho)^{1/2},$$
for some functions $g(z), h(z)$ analytic at $\rho$, and
which are $\Delta$-analytic, that is, they can be analytically continued to the so-called $\Delta$-domain $\Delta=\{z\, | \, |z|< \rho +\varepsilon, z\neq \rho, |\arg(z-\rho)|>\phi\}$ for some $\varepsilon >0$ and $0<\phi<\pi/2$.
Therefore, applying the transfer theorem by Flajolet and Odlyzko~\cite{flajolet1990singularity} (see e.g. Theorem VI.3~\cite{flajolet2009analytic}): for $\alpha<\beta$
$$[z^i]\, \left((1-z/\rho)^{\alpha} + O(1-z/\rho)^{\beta}\right)= [z^i]\,(1-z/\rho)^{\alpha} +  O\left([z^i]\, (1-z/\rho)^{\beta}\right)$$
and the basic scaling (Chapter 6, \cite{flajolet2009analytic}):
$$[z^i]\, (1-z/\rho)^{1/2}=- \frac{1}{2\sqrt{\pi}}\, i^{-3/2}\, \rho^{-i}\, (1+O(1/i)),$$
we obtain  asymptotics for the coefficient $[z^i]\, f(z)$  of the form
$$[z^i]\, f(z)= \frac{h(\rho)}{2\sqrt{\pi}}\, i^{-3/2}\, \rho^{-i}\, (1+O(1/i)).$$

\section{Main Results}

There is a strong sense that the phase transition of the Bohman-Frieze process is qualitatively the same as that of the \ER process: while certain constants involved may differ, the behavior is the same.  In statistical physics, a group of processes with the same critical exponents is said to belong to the same \lq Universality Class\rq.  In that spirit, much of the work on the Bohman-Frieze process has been aimed at investigating its phase transition and proving results analogous to what we already know about the \ER phase transition.

Our results in this paper are in the barely sub- and supercritical regimes, $t_c \pm \eps$ (Theorems \ref{subcriticalthm}, \ref{supercriticalthm}, and \ref{smallpointsthm}).
Asymptotics involving both $n$ and $\eps$ in Theorems \ref{subcriticalthm} and \ref{supercriticalthm} are given by first fixing $\eps$, letting $n \to \infty$, then letting $\eps \to 0$.  Much is known for the \ER process where $\eps = \eps(n) \to 0$, but here we keep $\eps$ fixed relative to $n$.   A recent preprint \cite{bhamidi2011bohman} investigates component sizes inside the critical window.  
We mention previous work in each regime and the analogous \ER properties, then we give our new results.

\subsection{Critical Point }

Bohman and Frieze initiated the study of Achlioptas processes by exhibiting a process with $t_c >1$.  Spencer and Wormald \cite{spencer2007birth} and Bohman and Kravitz \cite{bohman2006creating} showed that the precise value of $t_c$ for any bounded-size Achlioptas processes (including the \BF process) can be expressed as the blow-up point of $s_1(t)$, the deterministic function around which the susceptibility concentrates.  The corresponding function for the \ER process is $s_1(t) = \frac{1}{1-t}$, which blows up at $t=1$.  For the \BF process, $t_c$ can be computed numerically and is approximately $1.176$.

\subsection{Growth of Early Giant}

In the \ER random process, the fraction of vertices in the largest components at time $t$ is $\rho(t) +o(1)$ whp, where $\rho(t)$ is the survival probability of a Poisson Galton-Watson branching process with mean $t$.  For $t \le 1$, $\rho(t) =0$ and for $t>1$, $\rho$ is the non-zero solution to the following equation:
\begin{equation*}
1- \rho = e^{-t \rho}.
\end{equation*}
In the barely super-critical regime, this gives $\rho (1+\eps) \sim 2\eps $ as $\eps\to 0$; i.e. the giant component grows linearly after the critical point with an initial rate of $2$.

For the Bohman-Frieze process, no formula is known for the size of the giant component at an arbitrary $t$, but in the barely super-critical regime, Janson and Spencer showed \cite{1005.4494} that there exist constants $\gamma, K$ so that whp:
\begin{equation*}
\gamma \eps - K \eps^{4/3} \le  |C_1(t_c+\eps)|/n \le \gamma \eps + K \eps^{4/3}.
\end{equation*}
This roughly says that the giant component of the \BF process grows linearly after the critical point with an initial rate of $\gamma$.  $\gamma$ can be computed numerically and is approximately $2.463$.  Their method involves approximating the survival probability of a multi-type Poisson branching process using the differential equation method and the first three moments of the susceptibility.

\subsection{Barely Subcritical Regime}
For the \ER process, in the barely subcritical regime, whp all components are of size $O(\log n)$ and are simple. Indeed, whp $|C_1(t_c -\eps)| = \Theta(\eps^{-2} \log n)$.  The distribution of the number of cycles is also known (see \cite{bollobás1984evolution}, \cite{bollobas2001random}).  In fact, Flajolet, Knuth and Pittel \cite{flajolet1989first} found the distribution of the length of the first cycle to appear in the \ER process.
For the \BF process, Spencer and Wormald \cite{spencer2007birth} showed that whp $|C_1(t_c - \eps)| = O(\log n)$ with a constant in the $O(\cdot)$ that depends exponentially on $\eps^{-1}$; in Section \ref{largestsubsec}, we will show that whp $|C_1(t_c - \eps)| = \Omega(\eps^{-2}\log n)$, and in Conjecture \ref{ubconj} we conjecture a matching upper bound.  We also describe the component structure at $t_c -\eps$:

\begin{thm}
\label{subcriticalthm}
At $t_c - \eps$, whp every component in the \BF process is simple.  The expected number of unicyclic components is $ \sim  \frac{1}{2} \log \frac{1}{\eps} $ as $\eps\to 0$, and the probability that there are no cycles is $\sim \sqrt{\eps} $.
\end{thm}

\subsection{Barely Supercritical  Regime}

In the \ER process, there is a symmetry between the subcritical graph and the supercritical graph with the giant removed (see \cite{luczak1990component}).  $C_2(t_c + \eps)$ therefore is very close in distribution to $C_1(t_c-\eps)$.  In particular, whp $C_2(t_c+\eps)$ is simple and $|C_2(t_c+\eps)| = \Theta(\eps^{-2} \log n)$.  While no corresponding symmetry theorem for the \BF process is known, we find  the size and structure of the smaller components:

\begin{thm}
\label{supercriticalthm}
Whp $ |C_2(t_c+\eps)| = \Theta ( \eps^{-2} \log n)$.  Moreover, whp all components apart from $C_1$ are simple.
\end{thm}

\subsection{Vertices in Small Components}
The differential equation method shows \cite{spencer2007birth} that in a bounded-size Achlioptas process, the proportion of vertices in components of size $i$, for  $i$ fixed relative to $n$, is whp concentrated around a deterministic function $x_i(t)$:
\begin{equation}
\frac{X_i \left( t \cdot \frac{n}{2} \right)}{n} \, =  \,  x_i (t) + o(1).\label{concentration}
\end{equation}

For the \ER random process, there is a precise formula for $x_i(t)$ (see Equation (\ref{erexact})), which yields the following asymptotics as $i \to \infty$: 
\begin{equation}
 \label{ersmalleq}
x_i(t_c \pm \eps) = \frac{C(\eps)}{\sqrt{2 \pi}} i^{-3/2} e^{ - \frac{1}{2} D(\eps) \eps^2 i  } (1+O(1/i)) .
\end{equation}
Here $C(\eps) = 1 + O(\eps)$ and $D(\eps) = 1 + O(\eps)$.

In the \BF process, because of the dependence between edges, we cannot derive an exact formula similar to Equation (\ref{erexact}).  Instead, we derive an asymptotic formula for $x_i(t)$ similar to (\ref{ersmalleq}). To this end, we use the generating function $P(t,z)=\sum_i x_i(t) z^i$, which is related to the higher moments of the susceptibility by $\frac{\partial^k P(t,z)}{\partial^k z}\Big|_{z=1} = s_k(t)$. In Section \ref{ODEanalysis} we show that $x_i(t)$ is determined by the following system of ODE's:
\begin{align*}
x_1^\prime(t) &= -x_1(t) - x_1^2(t) +x_1^3(t)\\
x_2^\prime(t) &= 2x_1^2(t) - x_1^4(t) - 2(1-x_1^2(t))x_2(t)\\
x_i^\prime(t) &=  \frac{i}{2} ( 1- x_1^2(t)) \sum_{k < i}  x_k(t) x_{i-k}(t) - i ( 1- x_1^2(t)) x_i(t),\quad i\ge 2
\end{align*}
and therefore $P(t,z)$
satisfies a non-homogeneous quasi-linear PDE:
\begin{equation*}
 \frac{\partial P(t,z)  }{\partial t} - z (1-x_1^2(t)) (P(t,z) -1) \frac{\partial P(t,z)}{\partial z} = z(z-1)x_1^2(t)
\end{equation*}
with the initial condition $P(0,z) =z$.
Its solution is implicitly defined by a single function. Applying the singular implicit functions theorem  (see e.g. Lemma VII.3. in \cite{flajolet2009analytic}) to that solution, we derive the following asymptotic formula for $x_i(t)$:

\begin{thm}
\label{smallpointsthm}
Let $x_i(t)$ be the solution of the above ODE's.  Then 
\begin{equation}
x_i(t_c \pm \eps) = C(\eps) i^{-3/2} e^{ - D(\eps) \eps^2  i} (1+O(1/i)),
\end{equation}
where $C(\eps) = c+ O(\eps)$ and $D(\eps) = d + O(\eps)$ for absolute constants $c,d$.
\end{thm}

Notice that Theorem \ref{smallpointsthm} is a statement about the asymptotic behavior of the infinite sequence of deterministic functions $\{x_i(t)\}_{i=1}^{\infty}$ (which does not involve $n$ at all), not about the random graph process.  In particular, we do not apply Theorem \ref{smallpointsthm} for $i=i(n)$.  Instead we will apply it for $i = O(\eps^{-2})$ in proving Theorem \ref{supercriticalthm}, for $\eps>0$ a constant with respect to $n$.

\vf

In the rest of the paper we prove our main results; in Section \ref{subsec}, Theorem \ref{subcriticalthm}; in Section \ref{supsec}, Theorem \ref{supercriticalthm}; and in Section \ref{smallsec}, Theorem \ref{smallpointsthm}.

\section{Subcritical Regime: Proof of Theorem 1}
\label{subsec}

\subsection{Complexity of Components}

Our results in the barely subcritical case state that, as in the \ER process, no complex connected component emerges in the \BF before $t_c - \eps$ for any fixed $\eps > 0$.  Thus the \BF rule not only delays the formation of a giant component but also the formation of bicyclic components.  Our proof makes a connection between the rate at which unicyclic components are formed and the susceptibility of the graph.

To prove Theorem \ref{subcriticalthm} we first prove the following lemma concerning the number of cyclic components in the subcritical \BF process:

\begin{lem}
\label{sublemma}
At time $t_c -\eps$, the number of cyclic components in the \BF process is asymptotically Poisson with mean $\mu_\eps$, and  $ \mu _\eps \sim  \frac12\log \frac{1} \eps$  as $\eps \to 0$.
\end{lem}

\begin{proof}
We first begin by fixing $\eps>0$ and conditioning on the concentration of the values of $X_1(j)$ and $S(j)$.  Spencer and Wormald \cite{ spencer2007birth} show that with error probability $O(\exp(-n^{1/5}))$ we have $\frac{X_1(t \frac{n}{2})}{n} = x_1(t) + o(1)$ and $S(j) = s(t) + o(1)$ for all $t \le t_c -\eps$.

At step $j$ of the \BF process, conditioned on the event that $X_1$ and $S$ are concentrated near their deterministic counterparts, the probability that an edge is added within an existing component is
\[ \left (1 - \frac{X_1^2(j)}{n^2} \right ) \sum_{C_i} \frac{| C_i|^2}{n^2} = \frac{1}{n}(1 - x_1^2(2j/n )) s(2j/n) + o(n^{-1}).\]
We run the process for $\frac{(t_c- \eps) n }{2}$ steps, and so the number of cyclic components created converges to a Poisson random variable (see Barbour, Holst, Janson \cite{barbour1992poisson}) with mean
\[ \mu_\eps = \frac{1}{n} \sum_ {j=1} ^{\frac{(t_c- \eps) n }{2}} (1 - x_1^2(2j/n))s(2j/n) + o(1), \]
where the $o(1)$ is the error in the approximation of $x_1$ and $s$.  We can approximate $\mu_\eps$ by an integral as $n \to \infty$, with $dt \sim \frac{2}{n} dj$:
\begin{equation}
\label{muepseq}
 \mu_\eps = \frac{1}{2} \int_{t=0}^{t_c- \eps} (1- x_1^2(t))s(t) \, dt + o(1).
\end{equation}

\vf

To compute this integral, we need an expression for $s(t)$.  From Spencer and Wormald \cite{spencer2007birth} we have
\[ s^{\prime}(t) =  x_1^2(t) + (1- x_1^2(t)) s^2(t) \]
with the initial condition $s(0) =1$.  In particular, this differential equation blows up at $t = t_c$.  As we are interested in the limit as $\eps \to 0$, we need the asymptotics of $s(t_c -\eps)$ as $\eps \to 0$.  We compute:
\[ \left (\frac{1}{s(t)}  \right ) ^\prime = \frac{-1} {s^2(t) }s^\prime (t) = \frac{-x_1^2(t) }{s^2(t)} - (1-x_1^2(t))  \]
\[ \to -( 1-x_1^2(t_c)) \text{ as } t \to t_c. \]
So $\frac{1}{s(t_c- \eps)} \sim \eps (1-x_1^2(t_c))$ and $s(t_c - \eps) \sim \frac{1}{\eps (1-x_1^2(t_c))}$.  Now we integrate the right-hand side of equation (\ref{muepseq}) and find $\mu_\eps \sim \frac{1}{2} \log \frac{1}{\eps}$ as $\eps \to 0$.

\end{proof}

We now complete the proof of Theorem \ref{subcriticalthm} and prove that whp there are no complex components at $t_c - \eps$.
At step $j$, the probability of creating a bicyclic component is
\[  \left (1 - \frac{X_1^2(j)}{n^2} \right ) \sum_{C_i, C_k \text{ unicyclic }} \frac{|C_i| |C_k|  }{n^2}. \]
The process at $t_c - \eps$ is subcritical so $|C_i| \le K(\eps) \log n$ for all $i$ whp, where $K(\eps)$ is a constant independent of $n$ (see \cite{spencer2007birth}),  and  from Lemma \ref{sublemma}, whp the number of unicyclic components is $\le \log n$ (or any other function $r(n) = \omega(1)$). The expected number of bicyclic components created before $t_c - \eps$ is therefore
\[ \le (t_c -\eps) \frac{n}{2} \frac{K^2(\eps) \log ^ 4 n}{n^2} = o(1), \]
so whp there are no complex components at $t_c - \eps$.

\subsection{Size of the Largest Subcritical Component}
\label{largestsubsec}

Theorem \ref{c1sublowerthm} provides the lower bound of $\Omega(\eps^{-2} \log n)$ for $|C_1(t_c -\eps)|$.

\begin{thm}
\label{c1sublowerthm}
There exists an absolute constant $K$ so that whp,
\begin{equation}
|C_1(t_c-\eps)| \ge K \eps^{-2} \log n.
\end{equation}
\end{thm}
\begin{proof}
We will use the so-called ``sprinkle'' from $t_c - 2 \eps$ to $t_c -\eps$; that is, we first run the \BF process to time $t_c - 2 \eps$, and then, conditioned on the resulting graph with certain properties that hold whp, we run the \BF process from that graph to time $t_c -\eps$.

We begin at time $t_c -2\eps$ and call vertices in components of size between $m_1 \eps^{-2}$ and $m_2 \eps^{-2}$ \textit{medium} vertices, where  $m_1<m_2$ are constants which will be determined later.  Call the subgraph of medium vertices $G_M$.  We run the process from time $t_c - 2 \eps$ to $t_c -\eps$ and show that a connected component of size $\ge K \eps^{-2} \log n$ emerges whp in $G_M$, for some constant $K$.  This suffices for a lower bound on $|C_1|$ for the entire graph.  We call each connected component in $G_M$ at time $t_c - 2 \eps$ a `node'.
Using Theorem \ref{smallpointsthm}  we know that
whp the number of nodes in $G_M$ is
\begin{eqnarray*}
&\sum_{i=m_1 \eps^{-2}}^{m_2 \eps^{-2}} (x_i(t_c -2\eps) +o(1) ) \frac{n}{i}
 \ge c(m_1, m_2) \eps^{3} n
\end{eqnarray*}
for some constant $c(m_1, m_2)$.
Each pair of nodes has at least $m_1^2 \eps^{-4}$ potential edges between them. So from $t_c - 2 \eps$ to $t_c -\eps$, since we add at least $\frac{\eps n}{20}$ random edges as the second in a pair, 
each pair of nodes is joined with probability at least ${ m_1^2 \eps^{-3}}{n^{-1}}/20$.  
Now we couple $G_M$ with an \ER random graph that has $N=c(m_1, m_2) \eps^{3} n$ nodes and the probability of joining two nodes is $\frac{d}{N}$ with $d$ at least ${ m_1^2 c(m_1, m_2)/20}$ a constant independent of $n$. 
From the original results of \erdos and \renyi \cite{erdős1960evolution} we know that whp a connected component of $\ge K(d) \log N$ nodes forms in the graph.
Since each node has $\ge m_1 \eps^{-2}$ vertices, whp we have a connected component with $\ge K(m_1,m_2) m_1 \eps^{-2} \log n$ vertices at time $t_c - \eps$.

\end{proof}

We conjecture a matching upper bound\footnote{Subsequent to this work the authors of \cite{bhamidi2011bohman} have made substantial progress on this conjecture, proving an upper bound of $O(\eps^{-2} (\log n)^4)$ on $|C_1(t_c-\eps)|$.}:
\begin{conj}
 \label{ubconj}
There exists a constant $K_u$ so that whp
\begin{equation}
|C_1(t_c-\eps)| \le K_u \eps^{-2} \log n.
\end{equation}
\end{conj}

\section{Supercritical Regime: Proof of Theorem 2}
\label{supsec}
First we shall prove the following lemmas in Sections~\ref{upperC2},~\ref{lowerC2}, which correspond to the first part of Theorem \ref{supercriticalthm} that $|C_2(t_c + \eps) | = \Theta(\eps^{-2} \log n)$.
\begin{lem}
\label{BFsuperupperlem}
There exists a  constant $K_u$ so that whp $|C_2(t_c + \eps) |\le K_u \eps^{-2} \log n$.
\end{lem}

\begin{lem}
\label{superlowerlem}
There exists a  constant $K_l$ so that whp $|C_2(t_c + \eps) |\ge K_l \eps^{-2} \log n$.
\end{lem}

In Section~\ref{outsidegiant}  we shall prove the second part of Theorem \ref{supercriticalthm} that all components apart from the giant component are simple.

\subsection{Upper bound of $|C_2(t_c + \eps) |$}
\label{upperC2}

In order to prove Lemma \ref{BFsuperupperlem} we  use the ``sprinkle'' from $t_c + \eps/2$  to $t_c +\eps$.
We first run the \BF process up to time $t_c + \eps/2$ and call the resulting graph $G$.
Janson and Spencer  \cite{1005.4494} show that whp $G$ contains a giant component $C_1$ of size $\ge y \eps n$ for some constant $y$. We  let $G\setminus C_1=G^{small}$.
Next we run the \BF process further up to time $t_c + \eps$ and want to show that there exists a constant $K_u$ so that whp each component of size $\ge K_u \eps^{-2} \log n$ in $G^{small}$ at time $t_c + \eps$ is connected to $C_1$, so that $|C_2(t_c + \eps)|\le K_u \eps^{-2} \log n $. To this end we let $X$ be the number of components of size $\ge K_u \eps^{-2} \log n$ in $G^{small}$ at time $t_c + \eps$  that are not connected to $C_1$ and will show that $\Pr[ X \ge 1] = o(1)$ by proving $ \E X \le o(1)$.

Now we write
\begin{equation}
\label{condeq}
 \E X = \sum _{V_i} \Pr[ B_{V_i} ] \cdot \Pr [ V_i \nleftrightarrow C_1 | B_{V_i}],
\end{equation}
where the $V_i$'s are all possible sets of vertices in $G^{small}$  at time $t_c + \eps$ of size larger than $K_u \eps^{-2} \log n$; $B_{V_i}$ is the event that $V_i$  forms a connected, isolated component in $G^{small}$ at time $t_c + \eps$; and $V_i\nleftrightarrow C_1$ is the event that component $V_i$ is not connected to the largest component at $t_c +\eps$. Since there can be at most $n$ isolated components in $G^{small}$  at time $t_c + \eps$ and so $\sum _{V_i} \Pr[ B_{V_i} ] \le n $, it suffices to prove that for any $V\subset G^{small}$ with $|V| \ge K_u \eps^{-2} \log n$,
 \begin{equation}\label{IneqVC1}
   \Pr [ V \nleftrightarrow C_1 | B_{V}] \le n^{-2},
 \end{equation}
for this implies that $\E X \le n^{-1}$ as desired.

In order to prove (\ref{IneqVC1}), we consider the following special pairs of edges that may appear at a step in the \BF process between $t_c +\eps /2$ and $t_c +\eps$: call a pair $(e,f)$ of edges {\em red} if the first edge $e$ does not join two vertices which are isolated at time $t_c + \eps/2$ and the second edge $f$ joins a vertex in $C_1$ to a non-isolated vertex in $G^{small}$.  If such a pair is presented at a step in the \BF process between $t_c +\eps /2$ and $t_c +\eps$, the second edge $f$ will necessarily be added according to the \BF rule. Call such a second edge $f$ a {\em red edge}.  We note that the number of red edges added between $t_c + \eps/2 $ and $t_c + \eps$  is distributed as ${\rm{Bin}}  (\frac{\eps n}{4}, p)$, where $p =  (1 - x_1^2(t_c + \eps/2)) \cdot \frac{ |C_1| (n- |C_1| - n x_1(t_c+\eps/2))  }{ n^2} (1+o(1))$.
But for an upper bound we shall add each red edge independently with probability $\frac{c \eps}{n}$ with $c$ chosen small enough so that whp the actual red edges added in the \BF process are a superset of these independent edges.  Note that the choice of red edges does not affect the other edges in the \BF process since it does not affect which vertices are isolated.

Now we return to (\ref{IneqVC1}) and note that  for any $V\subset G^{small}$ with $|V| \ge K_u \eps^{-2} \log n$,
\begin{equation}\label{VC1f1fE}
\Pr [ V \nleftrightarrow C_1 | B_{V}]  \le \Pr [ f_1 \dots f_E \notin G | B_{V}]
\end{equation}
where $f_1 \dots f_E$ are the edges between $C_1$ and the vertices of $V$ which were not isolated at $t_c+\eps/2$.
In order to bound $\Pr [ f_1 \dots f_E \notin G | B_{V}]$ from above, we just consider the probability that each $f_i$ is added with probability $\frac{c \eps}{n}$ as an independent red edge (as discussed above). 
To bound the number of possible red edges, 
we color each vertex  that was isolated at $t_c + \eps/2$ {\em orange} and leave the vertices that were not isolated at time $t_c + \eps/2$ uncolored, i.e. non-orange. We claim the following.
\begin{prop}\label{Propnonisol}
There exists a constant $K_u$ so that whp there is no connected component in $G^{small}$ at $t_c +\eps$ with $\ge  \frac{9}{10} K_u \eps^{-2} \log n$ orange  vertices and $< \frac{1}{10} K_u \eps^{-2} \log n$ non-orange nodes.  In particular, whp every connected component in the $G^{small}$ subgraph at $t_c +\eps $ of size $\ge K_u \eps^{-2} \log n$ has at least $\frac{K_u}{10} \eps^{-2} \log n$ non-orange vertices.
\end{prop}
By Proposition \ref{Propnonisol}, there are at least $ y\eps n \cdot \frac{K_u}{10} \eps^{-2} \log n=y n \cdot \frac{K_u}{10} \eps^{-1} \log n$ possible {\em red edges} that can join $V$ to $C_1$. So, we have
\begin{equation}\label{upperf1fE}
\Pr [ f_1 \dots f_E \notin G | B_{V}] \le \left( 1 - \frac{c\eps}{n} \right) ^{y n \frac{K_u}{10} \eps^{-1} \log n} \le n^{-2},
\end{equation}
for $K_u$ chosen large enough.  From (\ref{VC1f1fE}) and (\ref{upperf1fE}), we have (\ref{IneqVC1}) as desired.

\begin{proof}[Proof of Proposition \ref{Propnonisol}]
For notational convenience we let $A=K_u \eps^{-2} \log n$.
We pick an arbitrary vertex $v \in G^{small}$ and perform a breadth-first search (BFS) of its connected component  at $t_c + \eps$.  This method is used to analyze the size of components in the \ER process very precisely; here we do not have independence between the edges in the \BF process, but for this proposition we just need coarse bounds on conditional probabilities.

Let $M \ge \frac{9}{10} A$ be a constant and let $Y$ be a binomial random variable with success (which in the BFS means that a given vertex has an orange vertex as a child) probability $\frac{2}{3}$ and $M$ trials. As we will see below, it is useful to consider the probability $Q=\Pr [ Y> \frac{8}{9} M]$.
Using a Chernoff bound for a binomial (e.g. \cite{janson2000random}), there exists a constant $c>0$ so that $Q=\Pr [ Y> \frac{8}{9} M]\le \exp ( - c M )$. The latter can be made to be $ o(n^{-1})$ for suitably large $K_u$.

We claim that the probability that a particular child of a vertex in the BFS is orange is $< \frac{2}{3}$ conditioned on the previous history of the BFS. Assuming the claim,  the probability that the component starting from $v$ has $\ge \frac{9}{10} A$ orange vertices and $< \frac{1}{10} A$ non-orange vertices (as children in the BFS) is less than $Q\le o(n^{-1})$. Furthermore, there are $n$ possible vertices from which to begin the BFS, so by the union bound the probability that a component with $> \frac{9}{10} A$ orange vertices has $< \frac{1}{10} A$ non-orange vertices is $o(1)$. In other words, whp any connected component in $G^{small}$ of size $\ge A$ at $t_c + \eps$ has $\ge \frac{1}{10} A$ non-orange vertices, as desired.

To prove the claim, we note that without conditioning on the BFS, given that we add an edge connected to a fixed isolated vertex $v$ at time $t$, the probability that the second vertex is isolated is
\begin{equation}
 p_1= \frac{x_1(t) + (1-x_1^2(t))x_1(t) }{x_1(t) + (1-x_1^2(t))} + o(1),
\end{equation}
and given that we add an edge containing a fixed non-isolated vertex $v$ at time $t$, the probability that the second vertex is isolated is
\begin{equation}
 p_2=x_1(t) + o(1).
\end{equation}
 The relevant history of the BFS are the edges incident on $v$.  We are conditioning on edges present in the graph, and can break up the conditional probability into two cases: all the conditioned edges appear later in the process than the edge we consider, or at least one conditioned edge appears earlier.  In the first case, the probability that the edge is added to an isolated vertex is at most $p_1$, and in the second case the probability is $p_2$.

We can bound $x_1(t)$ by $e^{-t}$ since $x_1(0)=1$ and $x_1^\prime(t) = -x_1(t) -x_1^2(t) + x_1^3(t)$, and we can bound $t_c >1$ by the results of Bohman and Frieze\cite{bohman2001avoiding} and Spencer and Wormald\cite{spencer2007birth}.  With these bounds, $\max(p_1, p_2) <\frac 2 3$ and so the claim is proved.

\end{proof}

\subsection{Lower bound of $|C_2(t_c + \eps) |$}
\label{lowerC2}

It will turn out to be useful to have the following concentration lemma of a sum of indicator random variables:
\begin{lem}
\label{chebylem}
Let $Y = \sum_{i=1}^M Y_i$, where the $Y_i$'s are indicator random variables, and both $M$ and the $Y_i$'s can depend on an underlying parameter $n$.  Let $\mu = \E Y$ and assume $\mu \to \infty$ as $n \to \infty$.  Assume that for all $i \ne j$, $\E [ Y_i Y_j] \le \E Y_i \E Y_j ( 1+f(n))$, where $f(n) \to 0$.  Then $Y \sim \mu$ whp.  Similarly, if $\E [Y_i | Y_j = 1] \le \E Y_i (1+ f(n))$ for all $i \ne j$, then $Y \sim \mu$ whp.
\end{lem}

\begin{proof}
Using Chebyshev's inequality it is enough to prove that $\frac{ \var(Y)}{( \E Y)^2 } = o(1)$.
\begin{eqnarray*}
\frac{ \var(Y)}{( \E Y)^2 }
&=& \frac{\E Y + \sum_{i \ne j} \E Y_i Y_j - (\E Y)^2  }{  (\E Y)^2  }\\
&=& \frac{ \sum_{i \ne j} \E Y_i Y_j }{(\E Y)^2  } - 1 + o(1) \text{ since  } \E Y \to \infty\\
&\le& \frac{\sum_{i \ne j} \E Y_i Y_j  } {\sum_{i \ne j} \E Y_i \E Y_j} - 1+o(1) \\
&\le& \frac{ (1+f(n)) \sum_{i \ne j} \E Y_i \E Y_j}{\sum_{i \ne j} \E Y_i \E Y_j} - 1 + o(1)\\
&=& f(n) + o(1) = o(1)
\end{eqnarray*}
\end{proof}

In order to prove Lemma \ref{superlowerlem}. We use the ``sprinkle'' from $t_c - \eps$  to $t_c +\eps$.

We first run the \BF process up to time $t_c - \eps$ and divide the vertices into three parts.
\begin{itemize}
\item[(1)] \textit{Small} : vertices in components of size $< m_1 \eps^{-2}$
\item[(2)] \textit{Medium}: vertices in components of size between $m_1 \eps^{-2}$ and $m_2 \eps^{-2}$
\item[(3)] \textit{Large}: vertices in components of size $> m_2 \eps^{-2}$.
\end{itemize}
Call the respective subgraphs $G_S$, $G_M$ and $G_L$ on these vertices.
The constants $m_1 < m_2$ will be determined later in the proof.

\begin{prop}
\label{sizesprop}
 Whp at time $t_c -\eps$ the above sets of vertices satisfy:
\begin{itemize}
\item[(1)] $|G_M| \sim c_M \eps n$
\item[(2)] $|G_L| \le c_L \eps n$
\end{itemize}
where $c_M$ and $c_L$ are constants depending on $m_1$ and $m_2$.
\end{prop}
\begin{proof}
Using Theorem \ref{smallpointsthm} we compute $|G_M|\sim \sum_{i=m_1 \eps^{-2}}^{m_2 \eps^{-2}} (x_i(t_c-\eps)+o(1)) \, n \sim \sum_{i=m_1 \eps^{-2}}^{m_2 \eps^{-2}} ((c+ O(\eps)) i^{-3/2} e^{ - (d + O(\eps)) \eps^2  i}+o(1)) \, n \sim c_M \eps n$ for a constant $c_M$ depending on $m_1$ and $m_2$. For $|G_L|$,
we know from the concentration of susceptibility that for some constant $c$, at $t_c -\eps$, $\sum_{v} |K(v)|  \le \frac{cn}{\eps} $ whp, so the number of vertices in components larger than $ m_2 \eps^{-2}$ is less than $\frac{ \eps^2}{m_2} \frac{cn }{\eps} = \frac{c}{m_2} \eps n$.
\end{proof}

Next we will run the \BF process from time $t_c -\eps$ to $t_c +\eps$ and show that whp for an appropriately chosen constant $K_l$, two separate connected components of size $\ge K_l \eps^{-2} \log n$ form in the $G_M$ subgraph without joining any vertices in $G_L$.  If the giant component at $t_c + \eps$ contains vertices from $G_L$, then either of these two components serve as the lower bound for the second-largest component.  If not, the smaller of the two serves as the lower bound, since the two components are not joined to each other.

We run the \BF process as follows. Call a pair of potential edges $(e_i, f_i)$ green if $e_i$ does not connect two isolated vertices at time $t_c - \eps$ and $f_i$ connects two medium components.  Call all other pairs of edges brown.  A randomly selected pair of edges will have probability $q= (1-x_1^2(t_c-\eps)) \cdot c_m^2 \eps^2 +o(1)$ of being green.  Assign a uniform $[0,1]$ random variable independently to each pair of edges (ordered, with repetition, brown and green pairs).   At each step of the process we flip an independent coin.  With probability $q$ we pick a random green pair of edges and add its second edge to the graph.  With probability $(1- q)$, pick the brown pair of edges with the lowest random number assigned, add the appropriate edge according to the \BF rule, and discard the pair.  Since the edges added from the green pairs do not touch isolated vertices, their choice will not affect the choice of edge from the brown pairs.  Similarly, the choice of brown edges will not affect the choice of green edges.

Now we describe a random graph, the lower bound graph, coupled to the \BF process.  We add the second edge of the same green pairs as above (a random choice of ${\rm{Bin}}(m, q)$ medium-medium edges), calling these edges green edges, then add \textit{both} edges from all pairs (brown and green) whose random variable is $\le \frac{ M}{\binom n 2 ^2 }$, calling these edges brown edges.

The following facts hold about this coupling:
\begin{enumerate}
\item  The green edges in the lower bound graph are the same edges as the \BF process adds from green pairs; the edges from brown pairs in the \BF process are a subset of the brown edges in the lower bound graph.
\item  Every edge in the graph is added independently with probability $p_b \sim \frac{c_b \eps}{n}$ as a brown edge in the lower bound graph.  These edges are  independent of the green edges added.
\item  If a component of size $\ge K_l \eps^{-2} \log n$ forms from green edges in $G_M$ which is not connected to any large vertex in the lower bound graph, then there is also such a component in the \BF process.
\end{enumerate}

Now we analyze the lower bound random graph.  Using a standard transformation we assume that each potential green edge is added independently with probability $p_g = \frac{c_g \eps}{n}$.

\begin{prop}
\label{goodprop}
 There exist constants $K_l$ and $K_m$ so that whp more than $n^{0.8}$ components of size between $K_l \eps^{-2} \log n$ and $K_m \eps^{-2} \log n$ form in $G_M$.
\end{prop}

\begin{proof}
Call the connected components in $G_M$ at time $t_c - \eps$ `nodes'.  We show that whp more than $n^{0.8}$ trees of $k= \frac{K_l}{m_1} \log n$ nodes form in the lower bound graph with the addition of the random green edges.  Let $X$ be the number of such trees.  From Proposition \ref{sizesprop} we know that the number of nodes is between
\begin{align*}
N_l &= \frac{c_M \eps^3 n}{m_2}, \\
N_u &= \frac{c_M \eps^3 n}{m_1}.
\end{align*}
 The probability that two nodes join together with the addition of the random green edges depends on the size of the nodes and is between:
\begin{align*}
p_l  &= 1 - \left( 1- p_g  \right )^ {m_1^2 \eps^{-4}  } \sim \frac{ c_g m_1^2 \eps^{-3}}{n}, \\
p_u  &= 1 - \left( 1-  p_g \right )^ { m_2^2 \eps^{-4} } \sim \frac{ c_g m_2^2 \eps^{-3}}{n}.
\end{align*}
So using the exact formula for the number of trees on $k= \frac{K_l}{m_1} \log n$ nodes, we calculate
\begin{align*}
\E X &\ge  \binom {N_l}{k} k^{k -2} p_l^{k-1} \left( 1- p_u   \right ) ^{ (N_u -k) k +\binom k 2 - (k-1) } \\
&\sim \frac{ N_l^k k^{k-2} (c_g m_1^2 \eps^{-3})^{k-1}}{k! n^{k-1}  } \exp \left [  -  \frac{ c_g m_2^2 \eps^{-3}}{n} N_u k   \right  ]   \\
&\sim  n \cdot  \frac{ \left (c_g m_1^2 m_2  c_M   \right ) ^k \eps^3 } {c_g m_2^2\sqrt{2 \pi} k^{3/2}    } \exp \left [  -  \frac{ c_g m_2^2 }{m_1}  k   \right  ]    \\
&\sim  n \cdot   \frac{  \eps^3 } {c_g m_2^2\sqrt{2 \pi} k^{3/2}    }  \left( c_g m_1^2 m_2^{-1}  c_M \cdot \exp \left ( -  \frac{ c_g m_2^2 }{m_1}  \right  )   \right ) ^{\frac{K_l}{m_1} \log n },
\end{align*}
which, for $K_l$ small enough, is $\ge n^{0.8}$.  Since each node has size between $m_1 \eps^{-2}$ and $m_2 \eps^{-2}$, setting $K_m = \frac{m_2}{m_1} K_l$ finishes the proof.

We now use Chebyshev's inequality via Lemma \ref{chebylem} to show concentration.  Let $T_i$ be the indicator random variable that a  set $i$ of $k$ nodes is in fact an isolated tree.  Then the number of isolated trees of size $k$ is $\sum_ i T_i$.    If trees $i$ and $j$ overlap, then $cov(T_i, T_j)$ is negative since the two sets of vertices cannot simultaneously be isolated trees. If the two trees do not overlap then
\[ \E (T_i T_j)  \le \E T_i \E T_j ( 1- p_u)^{- k^2}.  \]
Since $  ( 1- p_u)^{- k^2} - 1 = O( n^{-1} \log^2 n)$, we use Lemma \ref{chebylem} to conclude that $\sum_i T_i \sim \E \sum _i T_i $ whp.

\end{proof}

Call the above components `good' components.  Now we show that whp at least two good components avoid $G_L$ and the rest of $G_M$ with the addition of the brown edges.
Fix a good component, call it $Y$.  First we bound the number of small vertices that are connected to $G_M$ and $G_L$, excluding $Y$, with the addition of the random brown edges.

\begin{prop}
\label{smallprop}
For some constant $c_{bs}$, whp the number of small vertices connected to $G_M$ or $G_L$ at time $t_c + \eps$ is $\le c_{bs} \eps n$.
\end{prop}

\begin{proof}

Let $F = G_M \cup G_L \setminus Y$.  $|F| \le (c_L + c_M) \eps n =: c_F \eps n$.  Starting with the vertices in $F$ we  run a breadth-first search to explore the small vertices connected to $F$.  Using Theorem \ref{smallpointsthm}, and approximating the sum $\sum_{i=1}^{m_1 \eps^{-2}} i x_i(t_c - \eps)$ by an integral, we see that the average component size of small vertices is $\le \frac{ c \sqrt{m_1}}{\eps}$ whp. Each brown edge is added independently with probability $\frac{c_b \eps}{n}$, so we can bound above the size of the F-connected set of small vertices with a branching process in which each vertex has ${\rm{Pois}}(c_b \eps)$ descendant nodes, each of which has a number of vertices distributed as the size of a randomly chosen small node; in particular, the mean number of descendants in one node is $\le \frac{ c \sqrt{m_1}}{\eps}$. If $Z_i$ is a random variable representing the total number of descendants of one vertex, then $\E Z_i = \sqrt{m_1}c c_b <0.9 $ for a small enough choice of $m_1$. For the branching process to reach size $c_F \eps n +t \eps n$, we must have
\[\sum_{i=1}^{c_F \eps n +t \eps n} Z_i \ge t \eps n \]
Since $\E Z_i < 0.9$, for $t > 10 c_F$, a Chernoff bounds shows that the probability of this occurring is exponentially small.  Letting $c_{bs} = 12 c_F$ completes the proof.
\end{proof}

\begin{prop}
 The probability that component $Y$ does not join any vertex in  $G_L$ with the addition of the brown edges is $\ge n^{-0.1}$.
\end{prop}
\begin{proof}
$Y$ has size  between $K_l \eps^{-2} \log n$ and $K_l \frac{m_2}{m_1} \eps^{-2} \log n$.  The vertices that are potentially dangerous - vertices that could join $Y$ to $G_L$ - are the large vertices, the other medium vertices and the small vertices which are connected to either of the two sets.  Whp the total number of these dangerous vertices is $\le (c_{bs} + c_ L  + c_M) \eps n$.

The probability that $Y$ avoids all the  dangerous vertices with the addition of the random brown edges is therefor
\begin{align*}
&\ge   \left( 1-  \frac{c_b \eps}{n}   \right ) ^ { (c_{bs} + c_ L  + c_M) \eps n \cdot |Y|    }   \\
&\ge   \left( 1-  \frac{c_b \eps}{n}   \right ) ^ { (c_{bs} + c_ L  + c_M) \eps n  K_l \frac{m_2}{m_1} \eps^{-2} \log n    }   \\
&\sim  n^{ - c_b (c_{bs} + c_ L  + c_M) K_l  \frac{m_2}{m_1}  } \ge n^{-0.1}.
\end{align*}
for $K_l$ a small enough constant.
\end{proof}

\begin{prop}
 Whp at least two good components survives the brown edges without connecting to $G_L$ or the rest of $G_M$.
\end{prop}
\begin{proof}
The expected number of good components that survive the brown edges is $\ge n^{0.7}$.  We bound the variance.  Fix two good components, $Y$ and $Z$.  Let $X_Y$ and $X_Z$ be the indicator random variables that the respective components survive the brown edges.  From Lemma \ref{chebylem}, it is enough to show that $\E [X_Y | X_Z =1] = \E X_Y (1 + o(1))$.  Conditioning on $Z$ surviving, the number of edges that cannot be present changes by $|Z| \cdot |Y| \le K_l^2 \frac{m_2^2}{m_1^2} \eps^{-4} \log^2 n$, and each factor is $(1 - \Theta ( \frac{\eps}{n}))$, for a total change of $(1 + o(1))$ in the probability.
\end{proof}

So whp at least two of the components larger than $K_l \eps^{-2} \log n$ in $G_M$ avoid $G_L \cup G_M$ and all vertices in $G_S$ connected to them. This provides the lower bound for the second-largest component.

\subsection{Structure of Small Components}
\label{outsidegiant}

We prove the second part of Theorem \ref{supercriticalthm} that at $t_c +\eps$, all components apart from the giant component are simple.  Define the $L$-restricted susceptibility of $G$ to be
\[ S_L(G) = \frac{1}{n} \sum _{|C_i| \le L} |C_i|^2. \]
Note that $S_L(G) \le L$.  At step $j$ the probability that an edge is added within a component of size $\le L$ is
\[ \left (1- \frac{  X_1(j)^2}{n^2} \right) \sum_{|C_i|\le L} \frac{|C_i|^2}{n^2} = \left(1- \frac{  X_1(j)^2}{n^2} \right) \frac{S_L(j)}{n}   \le \frac{L}{n}.\]
The expected number of cyclic components of size $\le L$ created before time $t_c +\eps$ is $\le (t_c + \eps) \frac{n}{2} \frac{L}{n} \le L$ and so whp $\le L \log n$ such components are created. The probability a complex component of size $\le L$ is created at step $j$ is
\[ \le  \left (1- \frac{  X_1(j)^2}{n^2} \right) \sum_{C_i, C_k \text{ cyclic}, |C_i|, |C_k|\le L} \frac{|C_i| |C_k|}{n^2} \le \frac{L^2}{n^2} \cdot Z(j), \]
where $Z(j)$ is the number of cyclic components of size $\le L$ at step $j$.  By the above $Z(j) \le L \log n$ for all $j$ whp, so the expected number of complex components of size $\le L$ created by time $t_c +\eps $ is $\le \frac{L^3 \log n}{n}$.  If we set $L = K_u \eps^{-2} \log n$, this is $o(1)$, and so whp there are no complex components of size $\le K \eps^{-2} \log n$ at $t_c + \eps$, and therefore no complex components other than the giant component.

\section{Vertices in Small Components}
\label{smallsec}

Before proving Theorem \ref{smallpointsthm}, we give in Section \ref{ERprocess} a standard proof of the corresponding \ER result for comparison and another proof of it to illustrate a method based on a quasi-linear PDE and singularity analysis, which we will use in Section \ref{ODEanalysis} to prove  Theorem \ref{smallpointsthm}.

\subsection{The \ER Process}
\label{ERprocess}

If $\eps$ is fixed, and either positive or negative, we have the following exact formula (see e.g. Chapter 11 of \cite{alon2000probabilistic}):
\begin{equation}
\label{erexact}
x_i(t_c +\eps) = \lim_{n \to \infty} \Pr[|K(v)| = i ] = \frac{e^{- (1+\eps)i} ((1+\eps)i)^{i-1}   }{  i!}.
\end{equation}
Now using Stirling's formula we compute asymptotics in $i$:
\begin{align*}
 x_i(t_c +\eps) 
&= \frac{1} {\sqrt {2 \pi}} \frac{1}{1+\eps} i^{-3/2}  e^{-(\eps - \log (1+ \eps))i} \, (1+O(1/i)).
\end{align*}
As $\log(1+\eps)=\eps -\eps^2/2 +O(\eps^3)$, we have
\begin{equation*}
 x_i(t_c +\eps) = \frac{1 + O(\eps)} {\sqrt {2 \pi}}  i^{-3/2}  e^{-(\frac{\eps^2}{2} + O(\eps^3))i} \, (1+O(1/i)),
\end{equation*}
so
\begin{equation}
 \label{erxisstatement}
x_i(t_c +\eps) = C(\eps) i ^{-3/2}  e^{-d(\eps) \eps^2 i} \, (1+O(1/i)),
\end{equation}
where $C(\eps) = \frac{1}{\sqrt {2 \pi}} + O(\eps)$ and $d(\eps) = \frac{1}{2} + O(\eps)$.

\vf

Now we prove the same fact as above, but using a different method, the one we will later use to prove Theorem \ref{smallpointsthm}.  In Section \ref{ODEanalysis}, we write the differential equations for the functions $x_i(t)$ for the \BF process, describing the proportion of vertices in components of size $i$.  For the \ER process these equations are:

\begin{equation}\label{ERODE}
 x_i^\prime(t) =  - i x_i(t) + \frac{i}{2}\sum_{k<i} x_k(t) x_{i-k}(t)
\end{equation}
for $i \ge 1$, with initial conditions $x_1(0) =1$ and $x_i(0) = 0$ for $i \ge 2$.   We can define a function of two variables,
\begin{equation}
\label{ERptzdef}
 P(t,z) = \sum_{i \ge 1} x_i(t) z^i.
\end{equation}
Our goal is to extract the asymptotic behavior of the coefficient $x_i(t)$ of $P(t,z)$, using singularity analysis.

Multiplying by $z^i$ both sides of (\ref{ERODE}) and summing over $i$, since $\frac12 \frac{\partial [ P(t,z)^2]}{\partial z}= P(t,z) \frac{\partial P(t,z)}{\partial z}$,  we find a homogeneous quasi-linear PDE for the \ER  process:
\begin{equation}
\label{ERdpdt}
 \frac{\partial P(t,z)  }{\partial t} - z (P(t,z) -1) \frac{\partial P(t,z)}{\partial z}=0
\end{equation}
with an initial condition $P(0,z)=z$.  Using the method of characteristic curves (see for example, Section 3.2 of \cite{evans2010partial}), we note that a solution to (\ref{ERdpdt}) defines a surface $y= P(t,z)$ in $(t,z,y)$-space, and we obtain ODE's:
\begin{equation}
 \frac{dz}{dt} = -z(y-1) , \, \, \frac{dy}{dt} = 0
\end{equation}
with initial values $z(0) = z_0$ and $y(0) = y_0$, with $y_0 = z_0$ from the initial condition of the PDE.  We can solve these ODE's to get:
\begin{equation}
\label{yety}
 z(t) = ye^{t-ty }
\end{equation}
with $y= y_0 = z_0$ fixed.  We write:
\begin{equation}
 \label{zfyeq}
z = F(t,y) \text { with } F(t,y) = ye^{t-ty }.
\end{equation}

We now define a critical point with respect to (\ref{zfyeq}):
\begin{equation}
\label{tcrdef}
 t_{cr} =  \inf \Big\{ t: \frac{ \partial F(t,y) }{\partial y   }(t,1) =0   \Big\}.
\end{equation}
Note that for all $t$, $(t, 1,1)$ satisfies (\ref{yety}). Let us give some intuition for (\ref{tcrdef}).  We know that $P(t,1)=\sum x_i(t)$  amounts to the proportion of vertices in ``small'' components and that $P(t,1)=1$ for  $t<1$ and $P(t,1)<1$ for $t>1$. Furthermore, these "small" component are typically trees. Thus, loosely speaking, for fixed $t$, $y=P(t,z)$ would be equal to the generating function for vertex-rooted trees. From the generating function theory for trees (see e.g. \cite{flajolet2009analytic}), generating functions for various rooted trees exhibit a so-called square-root type singularity, which means that there is a value $\tau$ where the first derivative of the inverse function of $y$ vanishes. Thus we define $ t_{cr} $ as the smallest $t$ where the first derivative of the inverse function of $y$ vanishes as $z$ approaches 1.

For fixed $t$, we can write $z$ as a function of $y$, $z=F(y)= F(t,y)$.  If we invert and write $y$ as a function of $z$, we find a singularity where $\frac{d F}{d y} =0$.  Let $\tau$ be the $y$ so that $ \frac{d F}{d y}(\tau) =0$, and $\rho = F( \tau)$.  We find the Taylor expansion of $F(y)$ around $\tau$.
\begin{equation}
 z= \rho + F^\prime (\tau) (y- \tau) + \frac{1}{2} F^{\prime \prime} (\tau) (y - \tau)^2 + O((y- \tau)^3).
\end{equation}
We rewrite and solve for $y$, using the fact that $F^\prime (\tau) =0$ and $y$ increases along the real $z$-axis (since the coefficients of $y$ are non-negative):
\begin{align*}
 y(z) &= \tau - \sqrt {\frac{-2 \rho }{F^{\prime \prime }(\tau) } }\sqrt{1- \frac{z}{\rho} } +O\left(1- \frac{z}{\rho}\right).
\end{align*}

In summary, for each fixed $t$, the function $y=P(t,z)$ has the singular expansion around $\rho(t)$ of the form
\begin{align*}
 P(t,z) &= \tau(t) - \sqrt {\frac{-2 \rho(t) }{F^{\prime \prime }(\tau(t)) } }\sqrt{1- \frac{z}{\rho(t)} } +O\left(1- \frac{z}{\rho(t)}\right).
\end{align*}

Now using the transfer theorem (Theorem VI.3~\cite{flajolet1990singularity}) and the basic scaling (Chapter 6, \cite{flajolet2009analytic}), we extract the asymptotic behavior of the coefficient $x_i(t)$ of $P(t,z)$:
\begin{equation}
 x_i(t) =  \frac{c(t)}{2 \sqrt \pi} i^{-3/2} \gamma(t)^i \, (1+ O(1/i)),
\end{equation}
where
\begin{align*}
 c(t) &= \sqrt {\frac{-2 \rho(t) }{F^{\prime \prime }(\tau(t)) } },   \\
\gamma(t) &=  \frac{1 }{\rho(t)}.
\end{align*}

Now we analyze $c(t)$ and $\gamma(t)$.  Using (\ref{zfyeq}) we can calculate
\begin{equation}
\label{dzdyER}
 \frac{dF}{dy} = e^{t-ty} - tye^{t-ty}
\end{equation}
and
\begin{equation}
 \frac{d^2F}{dy^2} =t^2 y e^{t-ty} -2te^{t-ty}.
\end{equation}
Setting $\frac{dF}{dy}$ to $0$ gives $\tau(t) = \frac{1}{t}$ and
\begin{equation}
 \rho(t) = \frac{1}{t} e^{t-1}.
\end{equation}
From the definition of $t_{cr}$, we see that $t_{cr}=1$.
Since $\rho  (1)=1$, $\rho ^\prime (1)=0$ and $\rho ^{\prime \prime} (1) >0$, we can write
\begin{equation}
 \rho(1 + \eps) = 1 + \frac{1}{2} \eps^2 + O(\eps^3)
\end{equation}
and this gives
\begin{equation}
 \gamma(1+\eps) = 1 - \frac{1}{2} \eps^2 + O(\eps^3)
\end{equation}
and
\begin{equation}
\gamma(1+\eps)^i = e^{i \ln \gamma} = e^{- d(\eps)\eps^2 i }
 \end{equation}
where $d(\eps) = - \frac{1}{\eps^2}\ln (\gamma(1 +\eps) = \frac{1}{2} + O(\eps)$.
For $c(t)$ we have:
\begin{align*}
 c(1 +\eps) &= \sqrt {\frac{-2 (1 + \frac{1}{2}\eps^2 + O(\eps^3)) }{(1+\eps)e^\eps - 2(1+\eps)e^\eps }}   \\
&= \sqrt 2 (1+ O(\eps)).
\end{align*}
This recovers (\ref{erxisstatement}).

\subsection{The \BF Process}
\label{ODEanalysis}

\paragraph{Ordinary Differential Equations.}

We describe how the differential equations for the functions $x_i(t)$ for the \BF Process are derived.  The proof of convergence $\frac{X_i \left( t \, \frac{n}{2} \right)}{n} \, =  \,  x_i (t) + o(1)$ can be found in  \cite{bohman2006creating}, \cite{spencer2007birth}.

Begin with $X_1$, the number of isolated vertices in the graph.  The expected change in $X_1$ at step $j$ is (neglecting loops and repeated edges):
\begin{equation}
\E \Delta X_1(j) = -\frac{2 X_1(j)^2}{n^2} - 2 \left( 1- \frac{X_1(j)^2}{n^2} \right ) \frac{X_1(j)}{n}.
\end{equation}
Each edge added advances time by $\frac{2}{n}$, so we can approximate the derivative of $x_1(t)$ by dividing the expected change of $\frac{X_1}{n}$ by $\frac{2}{n}$:
\begin{equation}
\label{x1diffeq}
 x_1^\prime(t) = -x_1(t) - x_1^2(t) +x_1^3(t).
\end{equation}
By a similar calculation of the expected change in one round, we can write the differential equation for $x_2(t)$ and for general $x_i(t)$, $i>2$:
\begin{align}
\label{x2diffeq}
x_2^\prime(t) &= 2x_1^2(t) - x_1^4(t) - 2(1-x_1^2(t))x_2(t)    \\
\label{xidiffeq}
x_i^\prime(t) &=  \frac{i}{2} ( 1- x_1^2(t)) \sum_{k < i}  x_k(t) x_{i-k}(t) - i ( 1- x_1^2(t)) x_i(t).
\end{align}

\begin{prop}
\label{sumxisprop}
Let $t_c$ be the critical point of the \BF process.
The following hold concerning the functions $x_i(t)$, $i \ge 1$:
\begin{enumerate}
 \item[(i)] For $t< t_c$,
\[\sum_{i=1}^\infty x_i(t) = 1. \]
\item[(ii)] For $t>t_c$,
\[ \sum_{i=1}^\infty x_i(t) < 1. \]
\end{enumerate}
\end{prop}

\begin{proof}
(i)  Fix $\eps >0$.  From the concentration of susceptibility and the asymptotics of $s_1(t)$ \cite{spencer2007birth}, we know that for $t= t_c - \eps$,
\[ \frac{1}{n} \sum_ {v} |C(v)| =S_1(t\cdot  n/2)=s_1(t) +o(1)  \le \frac{K}{\eps}.\]
for some constant $K$ independent of $\eps$.

Thus for any $M$ there can be at most $\frac{n K \eps}{M}$ vertices in components of size $> M \eps^{-2}$. Choosing $M$ large enough, we see that
\[ \sum_{i=1}^{M \eps^{-2}} x_i(t_c - \eps) \ge 1- \frac{ K \eps}{M} + o(1). \]
Now let $M \to \infty$ to get the result.

\vf

(ii) Janson and Spencer\cite{1005.4494} show that at $t_c + \eps$ there is a component of size $\ge y \eps n$ for some constant $y$.  The vertices in this component are not counted towards any of the $x_i(t)$'s, so
\[\sum_{i=1}^\infty x_i(t_c+\eps) \le 1- y \eps \]
\end{proof}

\textit{Remark}: In fact, Riordan and Warnke \cite{1102.5306} have recently showed more: that $\sum_i x_i(t_c) =1$ and that for $t> t_c$, $\sum_i x_i(t) = 1 - y$, where the size of the giant at time $t$ is $\sim y n$.

\paragraph{Quasi-linear Partial Differential Equation.}
We again consider the generating function
\begin{equation}
 \label{PitBF}
P(t,z) = \sum_{i=1}^\infty x_i(t) z^i.
\end{equation}

We multiply  by $z^i$  the both side of (\ref{xidiffeq}) and sum over $i$. Since $\frac12 \frac{\partial P^2(t,z)}{\partial z}= P(t,z) \frac{\partial P(t,z)}{\partial z}$, using also (\ref{x1diffeq}),(\ref{x2diffeq}) we obtain a non-homogeneous quasi-linear PDE for the \BF process:
\begin{equation}
 \label{dpdtBF}
 \frac{\partial P(t,z)  }{\partial t} -z (1-x_1^2(t)) (P(t,z) -1) \frac{\partial P(t,z)}{\partial z} =z(z-1)x_1^2(t).
\end{equation}
Again we have the initial condition
\begin{equation}
 \label{BFinitcondit}
P(0,z) =z.
\end{equation}
If we let $f(t,x,P) =-z (1-x_1^2(t)) (P(t,z) -1) $ and $g(t,z) =  z(z-1)x_1^2(t)$, then we can rewrite (\ref{dpdtBF}) as
\begin{equation}
 \label{dpgfBF}
\frac{\partial P(t,z)  }{\partial t} + f(t,x,P)\frac{\partial P(t,z)}{\partial z} = g(t,z).
\end{equation}
A solution to this quasi-linear PDE defines an integral surface $y = P(t,z)$ in the $(t,z,y)$-space and we use the method of characteristics:
\begin{equation*}
 dt = \frac{dz}{f(t,z,y)} = \frac{dy}{g(t,z)}
\end{equation*}
or,
\begin{equation}
\label{BFODEs}
 \frac{dz}{dt} = f(t,z,y), \, \, \, \frac{dy}{dt} = g(t,z)
\end{equation}
with initial values $z_0$ and $y_0$ satisfying $z_0 = y_0$ due to (\ref{BFinitcondit}).  We solve the ODE's in (\ref{BFODEs}):
\begin{align}
\label{ztBF}
 z(t) &= z_0 \exp \left( - \int_0^t (1- x_1^2(s))(y(s)-1) \, ds \right )\\
\label{ytBF}
y(t) &= y_0 + \int_0^t z(s) (z(s)-1) x_1^2(s) \, ds.
\end{align}
Since $z_0 = y_0$ we write:
\begin{equation}
 \label{zztBF}
y(t) = z(t) \exp \left ( \int_0^t (1- x_1^2(s))(y(s)-1) \, ds  \right ) + \int_0^t z(s) (z(s)-1) x_1^2(s) \, ds.
\end{equation}
Let
\begin{equation*}
G(t,z,y)=-y + z\exp\Big(\int_{0}^{t}(1- x_1^2(s)) (y(s)-1)ds\Big) +\int_{0}^{t}z(s)(z(s)-1)x_1^2(s)ds.\label{eq.G}
\end{equation*}
Equation (\ref{zztBF})  implicitly defines $y$ as a function of $t$ and $z$ by
\begin{equation}
\label{zeqft}
G(t,z,y)=0.
\end{equation}
Notice that for all $t$, $y=z=1$ is a solution corresponding to the initial condition $z_0 = y_0 =1$.  We define
\begin{equation}
 \label{tcdefBF}
t_{cr} = \inf \Big\{t: \exists \text{ a solution to } G(t,z,y)=0: z\ge 1, y<1\Big\}.
\end{equation}
Proposition \ref{sumxisprop} shows that this is well-defined and that $t_{cr} = t_c$, where $t_c$ is the critical point of the \BF process.

\vt

\paragraph{Singularity Analysis.}
For a fixed $t$, we let $G(t,z,y)= G(z,y)$ and let
\begin{equation}
\label{eq.F}
F(z,y)=z \exp\Big(\int_{0}^{t}(1- x_1^2(s)) (y(s)-1)ds\Big) +\int_{0}^{t}z(s)(z(s)-1)x_1^2(s)ds,
\end{equation}
so that
\begin{equation*}
 G(z,y)=-y +F(z,y).
\end{equation*}
Note that $G(z,y)=0$ defines implicitly $y$ as a function of $z$.

Below we use the notation $G_z=\frac{\partial G}{\partial z}, G_y=\frac{\partial G}{\partial y},G_{yy}=\frac{\partial^2 G}{\partial^2 y}, F_z=\frac{\partial F}{\partial z}$ and $F_{yy}=\frac{\partial^2 F}{\partial^2 y}$.
We let $(\rho, \tau)$ be the solution to
\begin{eqnarray*}
G(\rho,\tau)&=&0\\
G_y(\rho,\tau)&=&0.
\end{eqnarray*}
Then, the singular implicit functions theorem (see e.g. Lemma VII.3. in \cite{flajolet2009analytic}) says that for any $\theta>0$,
there exists a neighborhood $D_{\theta}(\rho)$ of $\rho$ such that at every point $z\in D_{\theta}(\rho)-R_{\theta}(\rho)$,
(where $R_{\theta}(\rho)=\{z: z = \rho+ s (\cos \theta + i \sin \theta), s \ge 0\}$ the ray of angle $\theta$ emanating from $\rho$)
the equation $G(z,y)=0$ admits two analytic solutions $y_1(z)$ and $y_2(z)$
that satisfy
\begin{eqnarray*}
y_1(z) &=& \tau - \sqrt{\frac{2\rho G_z(\rho,\tau)}{G_{yy}(\rho,\tau)}} \Big(1-\frac{z}{\rho}\Big)^{1/2} +O\Big(1-\frac{z}{\rho}\Big)\\
y_2(z) &=& \tau + \sqrt{\frac{2\rho G_z(\rho,\tau)}{G_{yy}(\rho,\tau)}} \Big(1-\frac{z}{\rho}\Big)^{1/2} +O\Big(1-\frac{z}{\rho}\Big).
\end{eqnarray*}
However, since the coefficients of $y=y(z)$ (in the powers of $z$) are all non-negative, $y(z)$ is increasing along the real $z$-axis and therefore,
$y_1(z)$ is the right solution that we are looking for. For notational convenience let us use $y(z)$ instead of $y_1(z)$.
Since $G_z(z,y) = F_z(z,y)$ and $G_{yy}(z,y) = F_{yy}(z,y)$, we have
\begin{eqnarray*}
y(z)
&=&\tau - \sqrt{\frac{2\rho F_z(\rho,\tau)}{F_{yy}(\rho,\tau)}} \Big(1-\frac{z}{\rho}\Big)^{1/2} +O\Big(1-\frac{z}{\rho}\Big),
\end{eqnarray*}
In Proposition \ref{fzfyy} we show that $\frac{2\rho F_z(\rho,\tau)}{F_{yy}(\rho,\tau)}>0$.

In summary, for each fixed $t$,  the function $y=P(t,z)$ has the singular expansion around $\rho(t)$ of the form
\begin{eqnarray*}
P(t,z)
&=&\tau(t) - \sqrt{\frac{2\rho(t) F_z(\rho(t),\tau(t))}{F_{yy}(\rho(t),\tau(t))}} \left(1-\frac{z}{\rho(t)}\right)^{1/2} +O\left(1-\frac{z}{\rho(t)}\right).
\end{eqnarray*}
Now we use the transfer theorem (Theorem VI.3~\cite{flajolet1990singularity}) and the basic scaling (Chapter 6, \cite{flajolet2009analytic}) to extract the asymptotic formula of the coefficient $x_i(t)$ of $P(t,z)$:
\begin{equation}
\label{BFxipre}
 x_i(t) = \frac{c(t)}{2 \sqrt \pi} i^{-3/2} \gamma(t)^i \, (1+ O(1/i)),
\end{equation}
where
\begin{align}
\label{BFcteq}
 c(t) &= \sqrt {\frac{2 \rho(t) F_z(\rho(t),\tau(t)] }{F_{yy }(\rho(t),\tau(t)) } }   \\
\label{BFgteq}
\gamma(t) &=  \frac{1 }{\rho(t)}.
\end{align}

\vf

\paragraph{Existence and Properties of $\rho$.}
In order to simplify  notation in the proof we now make some definitions:
\begin{eqnarray*}
u(t,y)&=&\exp\Big(\int_{0}^{t}(1- x_1^2(s)) (y(s)-1)ds\Big),\\
v(t,z)&=&\int_{0}^{t}z(s)(z(s)-1)x_1^2(s)ds,\\
q(t,z)&=&\int_{0}^{t}(2z(s)-1)x_1^2(s)ds,\\
\alpha(t)&=&\frac{1-x_1^2(t)}{x_1^2(t)},\\
\beta(t)&=&\int_{0}^{t}(1- x_1^2(s))ds.
\end{eqnarray*}
Given a fixed $t$, we let
\begin{eqnarray*}
u(y)=u(t,y),\quad v(z)=v(t,z),\quad q(z)=q(t,z),\quad \alpha=\alpha(t), \quad \beta=\beta(t).
\end{eqnarray*}
We rewrite (\ref{eq.F}) as
\begin{eqnarray}\label{eq.Fnew}
F(z,y)
&=&z u(y) +v(z)
\end{eqnarray}
and differentiate  it  with respect to $y$:
\begin{eqnarray}
F_{y}(z,y)
&=& z u(y) \frac{\partial }{\partial y}\int_{0}^{t}(1- x_1^2(s)) (y(s)-1)ds\nonumber\\
&=& z u(y) \int_{0}^{t}(1- x_1^2(s))ds\nonumber\\
&=& \beta  z u(y).\label{eq.Fy}
\end{eqnarray}
Therefore, a solution $(\rho,\tau)$ to $G(z,y)=0$ and $G_y(z,y)=0$ (equivalently $y=F(z,y)$ and $F_y(z,y)=1$) must satisfy
\begin{eqnarray*}
y=F(z,y) =z u(y)+v(z),\quad
1=F_y(z,y)
&=& \beta z u(y).
\end{eqnarray*}
That is, 
\begin{eqnarray*}
z &=&\frac{1}{\beta u(y)},\quad 
y = \frac{1}{\beta}+v(z).
\end{eqnarray*}
We let $\rho$  be the smallest  positive real-valued solution of the function $y=y(z)$ that  satisfies the equation 
\begin{eqnarray*}
z \beta u(1/\beta+v(z)) =1.
\end{eqnarray*}
Pringsheim's theorem (Theorem IV.6 in \cite{flajolet2009analytic}) guarantees the existence of such a solution, since the coefficients of $y=y(z)$ (in the powers of $z$) are all non-negative.

\vf
Below  in Propositions \ref{fzfyy}, \ref{rhop} and \ref{rhopp}  we prove that $\frac{F_z(\rho,\tau)}{F_{yy}(\rho,\tau)}>0$, $\rho^\prime(t_{c}) =0$, and $\rho^{\prime \prime}(t_{c}) > 0$.

\begin{prop}
\label{fzfyy}
\begin{equation*}
 \frac{F_z(\rho,\tau)}{F_{yy}(\rho,\tau)}>0
\end{equation*}
\end{prop}

\begin{proof}
Differentiating (\ref{eq.Fnew}) with respect to $z$ we have
\begin{eqnarray*}
F_z(z,y)
&=&\frac{\partial }{\partial z} \Big(z u(y) +v(z)\Big)\nonumber\\
&=&u(y) +\frac{\partial}{\partial z}\Big(\int_{0}^{t}z(s)(z(s)-1)x_1^2(s)ds\Big)
\nonumber\\
&=&u(y) +q(z),
\end{eqnarray*}
and in particular, we have $F_z(\rho,\tau)=u(\tau)+q(\rho)$.

We differentiate (\ref{eq.Fy}) with respect to $y$:
\begin{eqnarray*}
F_{yy}(z,y)
&=&\beta z \frac{\partial }{\partial y}u(y)
=\beta^2 z u(y),
\end{eqnarray*}
and we have $F_{yy}(\rho,\tau)=\beta^2 \rho u(\tau)$.

Since  $u(\tau), q(\rho),\beta, \rho >0$, we have $F_z(\rho,\tau)>0$ and $F_{yy}(\rho,\tau)>0$.
\end{proof}

\begin{prop}
\label{rhop}
\begin{equation*}
 \rho'(t_{c})=0
\end{equation*}
\end{prop}

\begin{proof}
By $G_i(t,z,y)$ we mean the derivative of  $G(t,z,y)$  in $(t,z,y)$ with respect to the $i$-th variable.

We differentiate  $G(t,z,y)=G(t,z(t),y(t))$ with respect to $t$:
\begin{eqnarray*}
\frac{\partial }{\partial t}G(t,z(t),y(t)) &=&G_1(t,z,y)+G_2(t,z,y) z'(t)+G_3(t,z,y)y'(t),
\end{eqnarray*}
where
\begin{eqnarray}
G_1(t,z,y)&=&z u(t,y) (1- x_1^2(t)) (y-1) + z(z-1)x_1^2(t)\label{eq.G1}\\
G_2(t,z,y)&=&u(t,y)+q(t,z)\nonumber\\
G_3(t,z,y)&=&-1+\beta(t) z u(t,y).\label{eq.G3}
\end{eqnarray}

As $\rho(t),\tau(t)$ satisfies $0=G(t, \rho(t),\tau(t))$, given a fixed $t$, we have at $(\rho,\tau)=(\rho(t),\tau(t))$,
\begin{eqnarray}\label{rhoP}
0&=&G_1(t,\rho,\tau)+ G_2(t,\rho,\tau)\rho'(t)+G_3(t,\rho,\tau)\tau'(t).
\end{eqnarray}
But
\begin{eqnarray*}
G_1(t,\rho,\tau)&=&\rho u(\tau) (1- x_1^2(t)) (\tau-1) + \rho(\rho-1)x_1^2(t)\\
G_2(t,\rho,\tau)&=&u(\tau)+q(\tau)>0\\
G_3(t,\rho,\tau)&=&-1+\beta \rho u(\tau)=0,
\end{eqnarray*}
and so, from (\ref{rhoP}) 
\begin{eqnarray*}
\rho'(t) =-\frac{G_1(t,\rho,\tau)}{G_2(t,\rho,\tau)}= -\frac{\rho u(\tau) (1- x_1^2(t)) (\tau-1) + \rho(\rho-1)x_1^2(t)}{u(\tau)+q(\tau)}.
\end{eqnarray*}
Since $\rho(t_{c})=\tau(t_{c})=1$, we have
\begin{eqnarray*}
\rho'(t_{c})=0.
\end{eqnarray*}

\end{proof}

\begin{prop}
\label{rhopp}
\begin{equation*}
 \rho''(t_{c})>0
\end{equation*}
\end{prop}

\begin{proof}
By $G_{ij}(t,z,y)$ we mean the derivative of $G_{i}(t,z,y)$  in $(t,z,y)$ with respect to the $j$-th variable.

We differentiate  $G(t,z,y)=G(t,z(t),y(t))$ twice with respect to $t$:
\begin{eqnarray*}
&&\frac{\partial^2 }{\partial t^2}G(t,z,y)\\
&=&\frac{\partial}{\partial t}\Big(G_1(t,z,y)+G_2(t,z,y) z'(t)+G_3(t,z,y)y'(t)\Big)\\
&=&G_{11}(t,z,y)+G_{12}(t,z,y)z'(t)+G_{13}(t,z,y)y'(t)\nonumber\\
&&+\Big(G_{21}(t,z,y)+G_{22}(t,z,y)z'(t)+G_{23}(t,z,y)y'(t)\Big)z'(t)+G_{2}(t,z,y)z''(t)\nonumber\\
&&+\Big(G_{31}(t,z,y)+G_{32}(t,z,y)z'(t)+G_{33}(t,z,y)y'(t)\Big)y'(t)+G_{3}(t,z,y)y''(t).
\end{eqnarray*}

Given a fixed $t$, we have,  at $(\rho,\tau)$,
\begin{eqnarray*}
0&=&G_{11}(t,\rho,\tau)+G_{12}(t,\rho,\tau)\rho'(t)+G_{13}(t,\rho,\tau)\tau'(t)\nonumber\\
&&+\Big(G_{21}(t,\rho,\tau)+G_{22}(t,\rho,\tau)\rho'(t)+G_{23}(t,\rho,\tau)\tau'(t)\Big)\rho'(t)+G_{2}(t,\rho,\tau)\rho''(t)\nonumber\\
&&+\Big(G_{31}(t,\rho,\tau)+G_{32}(t,\rho,\tau)\rho'(t)+G_{33}(t,\rho,\tau)\tau'(t)\Big)\tau'(t)+G_{3}(t,\rho,\tau)\tau''(t).
\end{eqnarray*}
Since $G_{3}(t,\rho,\tau)=0, G_2(t,\rho,\tau)>0$, we have
\begin{eqnarray*}
\rho''(t)
&=&-\frac{G_{11}(t,\rho,\tau)+G_{12}(t,\rho,\tau)\rho'(t)+G_{13}(t,\rho,\tau)\tau'(t)}{G_{2}(t,\rho,\tau)}\nonumber\\
&&-\frac{\Big(G_{21}(t,\rho,\tau)+G_{22}(t,\rho,\tau)\rho'(t)+G_{23}(t,\rho,\tau)\tau'(t)\Big)\rho'(t)}{G_{2}(t,\rho,\tau)}\nonumber\\
&&-\frac{\Big(G_{31}(t,\rho,\tau)+G_{32}(t,\rho,\tau)\rho'(t)+G_{33}(t,\rho,\tau)\tau'(t)\Big)\tau'(t)}{G_{2}(t,\rho,\tau)}.
\end{eqnarray*}
Since $\rho(t_{c})=\tau(t_{c})=1$ and $\rho'(t_{c})=0$,
\begin{eqnarray}
\rho''(t_{c})
&=&-\frac{G_{11}(t_{c},1,1)}{G_{2}(t_{c},1,1)}\nonumber\\
&&-\frac{\Big(G_{13}(t_{c},1,1)+G_{31}(t_{c},1,1)+G_{33}(t_{c},1,1)\tau'(t_{c})\Big)\tau'(t_{c})}{G_{2}(t_{c},1,1)}.\nonumber\\\label{eq.rho2prime}
\end{eqnarray}

We differentiate $G_1(t,z,y)$ in (\ref{eq.G1}) with respect to the first variable:
\begin{eqnarray*}
G_{11}(t,z,y)
&=&-2z x_1(t)x_1'(t) (y-1)u(t,y)
+ z(1- x_1^2(t))^2 (y-1)^2 u(t,y)\\
&& + 2z(z-1) x_1(t)x_1'(t).
\end{eqnarray*}
Thus, given a fixed $t$, we have,  at $(\rho,\tau)$,
\begin{eqnarray*}
G_{11}(t,\rho,\tau)
&=&-2 \rho x_1(t)x_1'(t) (\tau-1)u(\tau)\nonumber\\
&&+ \rho (1- x_1^2(t))^2 (\tau-1)^2 u(\tau) + 2\rho(\rho-1) x_1(t)x_1'(t),
\end{eqnarray*}
and in particular, since $\rho(t_{c})=\tau(t_{c})=1$, we have
\begin{eqnarray}
G_{11}(t_{c},1,1)=0.\label{eq.G11}
\end{eqnarray}

We differentiate  $G_3(t,z,y)$ in (\ref{eq.G3}) with respect to  the first variable:
\begin{eqnarray*}
\frac{\partial }{\partial t}G_3(t,z,y)&=&G_{31}(t,z,y)+G_{32}(t,z,y)z'(t)+G_{33}(t,z,y)y'(t),
\end{eqnarray*}
where
\begin{eqnarray*}
G_{31}(t,z,y)&=&\beta'(t) z u(t,y)+ \beta(t) z u(t,y) (1- x_1^2(t)) (y-1)\\
G_{32}(t,z,y)&=&\beta(t)u(t,y)\\
G_{33}(t,z,y)&=&\beta(t)^2 z u(t,y).
\end{eqnarray*}
Thus, given a fixed $t$, we have,  at $(\rho,\tau)$,
\begin{eqnarray*}
0&=&G_{31}(t,\rho,\tau)+G_{32}(t,\rho,\tau)\rho'(t)+G_{33}(t,\rho,\tau)\tau'(t).
\end{eqnarray*}
and
\begin{eqnarray*}
G_{31}(t,\rho,\tau)&=&\beta'(t) \rho u(\tau)+ \beta  \rho u(\tau) (1- x_1^2(t)) (\tau-1)\\
G_{32}(t,\rho,\tau)&=&\beta u(\tau)>0\\
G_{33}(t,\rho,\tau)&=&\beta^2 \rho  u(\tau)>0,
\end{eqnarray*}
so
\begin{eqnarray*}
\tau'(t)&=&-\frac{G_{31}(t,\rho,\tau)+G_{32}(t,\rho,\tau)\rho'(t)}{G_{33}(t,\rho,\tau)}.
\end{eqnarray*}
In particular, since $\rho'(t_{c})=0$, we have
\begin{eqnarray}
\tau'(t_{c})&=&-\frac{G_{31}(t_{c},1,1)}{G_{33}(t_{c},1,1)}.\label{eq.tauprime}
\end{eqnarray}

From (\ref{eq.rho2prime}), (\ref{eq.G11}) and (\ref{eq.tauprime}) we get
\begin{eqnarray*}
\rho''(t_{c})
= \frac{G_{13}(t_{c},1,1)G_{31}(t_{c},1,1) }{G_{2}(t_{c},1,1)G_{33}(t_{c},1,1)}.
\end{eqnarray*}
Note that $G_2(t_{c},1,1)=u(1)+q(1)>0$,
$G_{31}(t_{c},1,1)=\beta'(t_{c}) u(t_{c},1)>0$ and $G_{33}(t_{c},1,1)=\beta^2(t_{c})  u(t_{c},1)>0$.
Finally, we differentiate $G_1(t,z,y)$ in (\ref{eq.G1}) with respect to the third variable:
\begin{eqnarray*}
G_{13}(t,z,y)
&=& z(1- x_1^2(t)) u(t,y)+  \beta(t) z(1- x_1^2(t)) (y-1)u(t,y).
\end{eqnarray*}
Thus, given a fixed $t$, we have,  at $(\rho,\tau)$,
\begin{eqnarray*}
G_{13}(t,\rho,\tau)
&=& \rho (1- x_1^2(t)) u(\tau)+  \beta  \rho(1- x_1^2(t)) (\tau-1) u(\tau)
\end{eqnarray*}
and so, 
\begin{eqnarray*}
G_{13}(t_{c},1,1)
&=& (1- x_1^2(t_{c})) u(t_{c},1) >0.
\end{eqnarray*}
Therefore, we have $\rho''(t_{c})>0$.

\end{proof}

\paragraph{Proof of Theorem \ref{smallpointsthm}.}

 From Equation (\ref{BFcteq}), and Propositions \ref{rhop} and \ref{rhopp},
\begin{eqnarray*}
 c(t_c+\eps) &=& c(t_{c}) +O(\eps)\\
 \rho(t_c+\eps) &=& 1+ \frac{\rho^{\prime \prime}(t_c)}{2}\eps^2   +O(\eps^3).
\end{eqnarray*}
If we let $d = \rho^{\prime \prime}(t_c)/2 > 0$ (see Proposition \ref{rhopp}), then since $\gamma(t)=1/\rho(t)$ we have
\begin{equation*}
 \gamma(t_c+\eps) = 1- d \eps^2 + O(\eps^3).
\end{equation*}
Let $c= \frac{c(t_{c})}{2 \sqrt \pi}$. Equation (\ref{BFxipre}) gives us:
\begin{equation}
 x_i(t_c + \eps) =  C(\eps)  i^{-3/2} e^{- D(\eps) \eps^2 i} \, (1 + O(1/i))
\end{equation}
where $C(\eps) = c+ O(\eps)$ and $D(\eps) = d + O(\eps)$.  This gives Theorem \ref{smallpointsthm}.

\vt
{\bf Acknowledgment.}
This research was carried out during the first author's visit to New York University as a Heisenberg Fellow of the German Research Foundation, 
and the second and third authors' visit to Berlin.

\bibliographystyle{abbrv}
\bibliography{wfcpbib.bib}

\end{document}